\documentclass[12pt,a4paper]{article}
\usepackage{amssymb,amsfonts,amsmath,amsthm,graphicx}

\theoremstyle{plain}
\newtheorem{theorem}{Theorem}[]
\newtheorem{definition}[theorem]{Definition}
\newtheorem{lemma}[theorem]{Lemma}
\newtheorem{corollary}[theorem]{Corollary}
\newtheorem{proposition}[theorem]{Proposition}
\newtheorem{example}[theorem]{Example}
\newtheorem{remark}[theorem]{Remark}

\begin{document}

\pagestyle{plain}

\pagenumbering{arabic}

\begin{center}\large{\bf Ring Theoretic Properties of Partial Crossed products and related themes }\end{center}

\vskip5mm

\begin{center}{\bf {\rm Laerte $Bemm^1$}}, {\rm Wagner $Cortes^2$},\end{center}

\begin{center} {\footnotesize $^{1}$ Departamento de Matem\'{a}tica\\
Universidade Estadual de Maring\'a\\
Maring\'a, PR, Brazil\\
Avenida Colombo, 5790\\
CEP 87020-900\\
e-mail: {\it lbemm2@uem.br}}
\end{center}

\begin{center}{\footnotesize $^{2}$  Instituto de Matem\'{a}tica\\
Universidade Federal do Rio Grande do Sul\\
91509-900, Porto Alegre, RS, Brazil\\
e-mail:  {\it wocortes@gmail.com}}
\end{center}

\date{}

\begin{abstract}
In this paper we work with unital twisted partial actions. We investigate ring theoretic properties of partial crossed products as artinianity, noetherianity, perfect property, semilocalproperty, semiprimary property and we also study the Krull dimension. Moreover, we consider  triangular matrix representation of partial skew group rings, weak and global dimensions of partial crossed products Also we study when  the partial crossed products are Frobenius and symmetric algebras.

\end{abstract}

\section*{Introduction}\

Partial actions of groups have been introduced in the theory of
operator algebras as a general approach to study $C^{*}$-algebras by
partial isometries (see, in particular, \cite{E1} and \cite{E3}),
and crossed products classically, as well-pointed out in
\cite{DES1}, are the center of the rich interplay between dynamical
systems and operator algebras (see, for instance, \cite{M1} and
\cite{Q1}). The general notion of (continuous) twisted partial
action of a locally compact group on a $C^{*}$-algebra and the
corresponding crossed product were introduced in \cite{E1}.
Algebraic counterparts for some notions  mentioned above were
introduced and studied in \cite{DE}, stimulating further
investigations, see for instance, \cite{LF},  \cite{CCF}, \cite{laz e mig}
and references therein. In particular, twisted partial actions of
groups on abstract rings and corresponding crossed products were
recently introduced in \cite{DES1}.

 In \cite{lmsw} and \cite{CM} the authors investigated ring theoretic properties of partial crossed products. In this article, we continue  the  investigation of  ring theoretic properties of partial crossed products as artinianity, noetherianity, semilocal  property, perfect property and semiprimary property. In these cases,  we study necessary and suffcient conditions for the partial crossed  products satisfy such conditions.  We study the Krull dimension of partial crossed products and we compare it with the Krull dimension of the base ring.  

We study triangular matrix representation of partial skew group rings  and we study global dimension, weak global dimension of partial crossed products. Moreover, we investigate when the partial crossed products are Frobenius  and symmetric algebras. 

This article is organized as follows: In the Section 1, we give some preliminaries and results that will be used during this paper.

In the Section 2,  we study necessary and sufficient conditions for the partial crossed products to be artinian, noetherian, semilocal, semiprimary and perfect, where we generalize the results presented in \cite{park1} and \cite{park2}.

In the Section 3, we only consider  partial actions of groups and we study the Krull dimension of partial skew group rings  and we compare it with the Krull dimension of the base ring. In these cases we extend the results presented in \cite{park2}.

In the Section 4, we define relative partial actions  of groups  on \linebreak  bimodules and  triangular matrix algebras, and we give some conditions for the existence of enveloping actions. Moreover, we study triangular matrix representation  of partial skew group rings.

In the Section 5,  we study global dimension, weak global dimension of partial crossed products. We finish the article by presenting conditions for partial crossed products to be Frobenius and symmetric algebras.

\section{ Preliminaries}\

In this section, we recall some notions about twisted partial actions on rings. More details can be found in \cite{DES1}, \cite{DES2} and \cite{DE}.



We begin with the following definition that is a particular case of   (\cite{DES2}, Definition 2.1).

\begin{definition} \label{def1}
An unital  \textit{twisted partial action} \index{twisted partial action} of a group
$G$ on a ring $R$ is a triple
\begin{center}
\vskip-1mm $\alpha = \big(\{D_g\}_{g \in G}, \{\alpha_g\}_{g \in G}, \{w_{g,h}\}_{(g,h) \in G\times G}\big)$,
\end{center}
\vskip-1mm where for each $g \in G$, $D_g$ is a two-sided ideal in
$R$ generated by a central idempotent $1_g$,  $\alpha_g:D_{g^{-1}} \rightarrow D_g$ is an isomorphism of
rings  and for each $(g,h) \in G \times G$, $w_{g,h}$ is an
invertible element of $D_g D_{gh}$, satisfying the
following postulates, for all $g,h,t \in G$:
\begin{itemize}
\item [$(i)$] \vskip-2.2mm $D_e = R$ and $\alpha_e$ is the identity map of $R$;
\item [$(ii)$] \vskip-2.2mm $\alpha_g(D_{g^{-1}} D_h)= D_g D_{gh}$;
\item [$(iii)$] \vskip-2.2mm $\alpha _g \circ \alpha _h (a)= w_{g,h} \alpha_{gh}(a) w_{g,h}^{-1}$, for all $a \in D_{h^{-1}} D_{h^{-1} g ^{-1}}$;
\item [$(iv)$] \vskip-2.2mm $w_{g,e}=w_{e,g}=1$;
\item [$(v)$] \vskip-2.2mm $\alpha_g(a w _{h,t}) w _{g,ht}= \alpha_g(a)w_{g,h} w _{gh,t}$, for all $a \in D_{g^{-1}} D_{h}D_{ht}$.
\end{itemize}
\end{definition}


\begin{remark}

 If $w_{g,h} = 1_g1_{gh}$, for all $g,h\in G$,  then we
have a partial action as defined  in
(\cite{DE}, Definition 1.1) and when $D_g = R$, for all $g\in G$,
we have that $\alpha$ is a twisted global action.
\end{remark}

Let $\beta = \big(T, \{\beta_g\}_{g \in G}, \{u_{g, h}\}_{(g, h)
\in G \times G}\big)$ be  a twisted global action of a group $G$
on a (non-necessarily unital) ring $T$ and $R$ an ideal of $T$
generated by a central idempotent $1_R$. We can restrict $\beta$
to $R$ as follows. Putting \linebreak $D_g = R\cap \beta_{g}(R) = R\cdot
\beta_{g}(R)$, $g\in G$, each $D_{g}$ has an identity element
$1_R\beta_{g}(1_R)$. Then defining
$\alpha_{g}=\beta_{g}|_{D_{g^{-1}}}$, for all $g\in G$, the items ($i$),
($ii$) and ($iii$) of \linebreak Definition \ref{def1} are satisfied.
Furthermore, defining $w_{g, h} = u_{g,
h}1_R\beta_g(1_R)\beta_{gh}(1_R)$, $g,h\in G$, the items ($iv$),
($v$) e ($vi$) of Definition \ref{def1} are also satisfied. So, we obtain a twisted partial
action of $G$ on $R$.


The following definition appears in (\cite{DES2}, Definition 2.2).

\begin{definition}\label{def2}
A twisted global action $\big(T, \{\beta_g\}_{g\in G},
\{u_{g,h}\}_{(g,h)\in G\times G}\big)$ of a  group $G$ on an
associative (non-necessarily unital) ring $T$ is said to be an
enveloping action \index{enveloping action} (or a globalization)
for a twisted partial action $\alpha$ of $G$ on a ring $R$ if
there exists a monomorphism $\varphi:R\rightarrow T$ such that,
for all $g$ and $h$ in $G$:
\begin{itemize}
\item [$(i)$] \vskip-2.2mm $\varphi(R)$  is an  ideal of $T$;
\item [$(ii)$] \vskip-2.2mm $T = \sum_{g\in G}\beta_g(\varphi(R))$;
\item [$(iii)$] \vskip-2.2mm $\varphi(D_g) = \varphi(R)\cap \beta_g(\varphi(R))$;
\item [$(iv)$] \vskip-2.2mm $\varphi\circ \alpha_g(a) = \beta_g\circ \varphi(a)$, for all $a\in D_{g^{-1}}$;
\item [$(v)$] \vskip-2.2mm $\varphi(aw_{g,h}) = \varphi(a)u_{g,h}$ and $\varphi(w_{g,h}a) = u_{g,h}\varphi(a)$, for all $a\in D_gD_{gh}$.
\end{itemize}
\end{definition}

In (\cite{DES2}, Theorem 4.1), the authors studied  necessary and
sufficient conditions for a twisted partial action $\alpha$ of a
group $G$ on a ring $R$ has an enveloping action. Moreover, they
studied which rings satisfy such conditions.

Suppose that $(R, \alpha, w)$ has an enveloping action
$(T,\beta,u)$. In this case, we may assume that $R$ is an ideal of
$T$ and we can rewrite the conditions of the Definition \ref{def2}
as follows:
\begin{itemize}
\item [$(i')$] \vskip-3mm $R$  is an  ideal of $T$;
\item [$(ii')$] \vskip-3mm $T = \sum_{g\in G}\beta_g(R)$;
\item [$(iii')$] \vskip-3mm $D_g = R\cap \beta_g(R)$, for all $g\in G$;
\item [$(iv')$] \vskip-3mm $\alpha_g(a) = \beta_g(a)$, for all $x\in D_{g^{-1}}$ and  $g\in G$;
\item [$(v')$] \vskip-3mm $aw_{g,h} = au_{g,h}$ and $w_{g,h}a = u_{g,h}a$, for all $a\in D_gD_{gh}$ and $g, h \in G$.
\end{itemize}

Given a twisted partial action $\alpha$ of a group $G$ on a ring
$R$,  we recall from (\cite{DES1}, Definition 2.2) that the
\emph{partial crossed product}\index{partial crossed product}
$R*_{\alpha, w}G$ is the direct sum
\begin{center}
\vskip-1mm $\displaystyle\bigoplus_{g\in G}D_g\delta_g$,
\end{center}
\vskip-1mm
where $\delta_g$'s are symbols, with the usual addition and multiplication defined by
\vskip-5mm
$$(a_g\delta_g)(b_h\delta_h) = \alpha_g(\alpha^{-1}_g(a_g)b_h)w_{g,h}\delta_{gh}.$$
\vskip2mm

By (\cite{DES1}, Theorem 2.4), $R*_{\alpha, w}G$ is an associative
ring whose identity is $1_R\delta_1$. Note that we have the
injective morphism $\phi: R\rightarrow R*_{\alpha, w}G$, defined
by $r\mapsto r\delta_1$ and we can consider $R*_{\alpha, w}G$ as
an extension of $R$. .

The following definition appears in (\cite{lmsw}, Definition 4.13).

\begin{definition} Let $\alpha$ be an unital  twisted partial action. We say that $\alpha$  is of finite type if there exists a finite subset $\{g_1, g_2, \cdots ,g_n\}$ of $G$ such that \begin{center} $\sum_{i=1}^n D_{ggi} =R$,\end{center} for all $g\in G$. \end{definition}

 It is convenient to point out that in the same way as in (\cite{laz e mig}, Proposition 1.2), we can prove that an unital  twisted partial action $\alpha$ of a group $G$ on an unital ring $R$ with an enveloping action $(T,\beta,u)$ is of finite type if and only if there exists $g_1,\cdots, g_n\in G$ such that $T=\sum_{i=1}^n \beta_{g_{i}} (R)$  if and only if $T$  has an identity element. 
 
  It is convenient to remind that two rings $A,B$ are Morita equivalent rings if their module categories are equivalent, see (\cite{AF}, pg 251),  for more details.  Now, we have the  following result which is a particular case of  ({\cite{DES2}, Theorem 3.1).
 
 \begin{proposition} \label{morita} Let $\alpha$ be an unital twisted partial action of a group $G$ on a ring $R$ with enveloping action $(T,\beta,u)$. If $\alpha$ is of finite type, then the rings $R*_{\alpha,w}G$ and $T*_{\beta,u}G$ are Morita equivalent rings.
 \end{proposition}

Let $\alpha$ be a twisted partial action of a  group $G$ on a ring
$R$. A subring (resp. ideal) $S$ of $R$ is said to be $\alpha$\textit{-invariant} if
$\alpha_g(S\cap D_{g^{-1}}) \subseteq S\cap D_g$, for all $g\in
G$. This is equivalent to say that $\alpha_g(S\cap D_{g^{-1}}) = S\cap D_g$, for all $g\in
G$. In this case the set
\begin{center}
$S*_{\alpha, w}G=\left\{\sum_{g\in G}a_g\delta_g\, |\, a_g\in S\cap D_g\right\}$
\end{center}
is a subring (resp. ideal) of $R*_{\alpha, w}G$.

If $I$ is an $\alpha$-invariant ideal of $R$, then the twisted partial action $\alpha$ can be extended to an twisted partial action $\overline{\alpha}$ of $G$ on $R/I$ as follows: for each $g\in G$, we define $\overline{\alpha}_g:D_{g^{-1}}+I\longrightarrow D_g+I$, putting $\overline{\alpha}_g(a+I)=\alpha_{g}(a)+I$, for all $a\in D_{g^{-1}}$, and for each $(g,h)\in G\times G$, we extend each $w_{g,h}$ to $R/I$ by \linebreak $\overline{w}_{g,h}=w_{g,h}+I$. In this case, $I*_{\alpha, w}G$ is an ideal of $R*_{\alpha, w}G$ and there exists a natural isomorphism between $(R*_{\alpha, w}G)/(I*_{\alpha, w}G)$ and $(R/I)*_{\overline{\alpha}, \overline{w}}G$. Moreover, when  $(R,\alpha,w)$ has enveloping action $(T,\beta,u)$, then by similar methods presented in  Section 2 of \cite{laz e mig}, we have that $(T/I^{e}, \overline{\beta},\overline{u})$ is the enveloping action of $(R/I,\overline{\alpha},\overline{w})$, where $I^e$ is the $\beta$-invariant ideal such that  $I^{e}\cap R=I$. Following (\cite{laz e mig}, Remark 2.3]) we have that the Jacobson radical $J(R)$  of $R$ is $\alpha$-invariant and so,  we can extend $\alpha$ to $R/J(R)$.

For convenience, we recall some concepts of ring theory that we will freely use
 in this paper (see \cite{lam1} for more details).

\begin{definition}
Let $S$ be a ring. Then we have the following definitions.
\begin{itemize}
\item[(i)] $S$ is semiprimitive if the Jacobson radical
$J(S)$ is zero.
\item[(ii)] $S$ is semilocal if $S/J(S)$ is
semisimple.
\item[(iii)]  $S$ is right perfect if $S$ is semilocal
and $J(S)$ is $T$-nilpotent, or equivalently, $S$ is right perfect if $S$ satisfies $DCC$
on principal left ideals (see (\cite{lam1}, Theorem 23.20])).
\item[(iv)] $S$ is semiprimary if $S$ is semilocal and $J(S)$ is
nilpotent.
\end{itemize}
\end{definition}

We finish this section with the following three results, where the first and second result appear in (\cite{jan}, Lemma 3) and  (\cite{park1}, Proposition 2.1), respectively.

\begin{proposition} \label{jan lema3}
Let $S$ be a ring with identity and $B$ a subring of $S$ with the
same identity such that $B$ is a left $B$-direct summand of $S$. If
$S$ is semilocal
then so is $B$.
\end{proposition}

\begin{proposition} \label{proposicao 2.1 park1}
Let $S$ be a ring with identity and $B$ a subring of $S$ with the
same identity. If $B$ is a left (resp. right) $B$-direct summand
of $S$, then for any right (resp. left) ideal $I$ of $B$, $IS\cap B=I$
(resp. $SI\cap B=I$).
\end{proposition}

As a particular case of Proposition \ref{proposicao 2.1 park1} we
have the following lemma.

\begin{lemma}
Let $\alpha$ be an unital  twisted partial action of a group $G$ on $R$ and
$B$ an $\alpha$-invariant subring of $R$ such that $_{B} R=_{B}B\oplus _{B} C$. Then for any right ideal $I$ of $B*_{\alpha,w}G$
we have that $I(R*_{\alpha,w}G)\cap (B*_{\alpha,w}G)=I$.
\end{lemma}

\section{Ring Theoretic Properties of Partial Crossed Products}\

Throughout this section $R$ is an associative ring with an
identity element $1_R$, $G$ is a group and  $\alpha =
\big(\{D_g\}_{g \in G}, \{\alpha_g\}_{g \in G}, \{w_{g,h}\}_{(g,h)
\in G\times G}\big)$ is an unital  twisted partial action of $G$ on $R$
such that $\alpha$ does not  necessarily admit enveloping action, unless otherwise stated.

Let $x=\sum a_gu_g\in R*_{\alpha,w}G$. Then the  \emph{support} of $x$
is  the set
\begin{center}
$supp(x)=\{g\in G\, |\, a_g\neq 0\}$.
\end{center}

For each  subgroup  $H$ of $G$,  we can restrict $\alpha$ to $H$, i.e., if
 we define $(R,\alpha_H,w_H)$ as  $\alpha_H=\{\alpha_h: D_{h^{-1}}\longrightarrow
D_h\, | \, h\in H\}$  and  $w_H=\{w_{l,m}\}_{l,m\in H}$ then we have a twisted partial of $H$ on $R$. Moreover, we have the partial crossed product
\begin{center}
$R*_{\alpha_H,w_H}H=\{\sum_{h\in H} a_h\delta_h\,|\, a_h\in D_h\}$
\end{center}
with usual sum and multiplication rule
$(a_h\delta_h)(b_l\delta_l)=\alpha_h(\alpha_{h^{-1}}(a_h)b_l)w_{h,l}\delta_{hl}$
and it is a subring of $R*_{\alpha,w}G$ with the same identity
$1_R$ of $R*_{\alpha,w}G$.

Let $A$ be the set
\begin{center}
$A=\{x\in R*_{\alpha,w}G \, |\, supp(x)\subseteq G - H\}$.
\end{center}
Since $supp(0)=\emptyset \subseteq G - H$, we have that $0\in A$
and obviously \begin{center}$A\cap (R*_{\alpha_H, w_H}H) =0$.\end{center}


The following result will be useful in this section.

\begin{lemma}\label{lema 1B}
Let $\alpha$ be an unital  twisted partial action of a group $G$ on $R$ and $H$ a subgroup of $G$.   Then $A=\{x\in R*_{\alpha,w}G \, |\,
supp(x)\subseteq G - H\}$ is a left and right
$(R*_{\alpha_H, w_H}H)$-submodule of $R*_{\alpha,w}G$, where the sum and
the scalar multiplication are those inherited from $R*_{\alpha,w}G$.
Moreover, in this case $ (R*_{\alpha_H,w_H}H)$ is a left and right
$(R*_{\alpha_H,w_H}H)$-direct summand of $R*_{\alpha,w}G$ and $A$ is the
complement.
\end{lemma}
\begin{proof}
It is not difficult to show that   $A$ is an  $(R*_{\alpha_H,w_H}H, R*_{\alpha_H,w_H}H)$-bimodule.  Finally, it is easy to
 see that $R*_{\alpha,w}G= (R*_{\alpha_H,w_H}H)\oplus A$ as a right
$(R*_{\alpha_H,w_H}H)$-module. The left side is completely
similar.
\end{proof}

Next, using Lemma \ref{lema 1B} we easily get the following result and we give a short proof for reader's convenience.

\begin{proposition} \label{proposição 10} Let $\alpha$ be an unital  twisted partial action of a group $G$ on a ring
 $R$ and $H$  a subgroup of $G$. If $R*_{\alpha,w}G$ is right artinian (resp. right noetherian, right perfect, semilocal), then $R*_{\alpha_H,w_H}H$ is right artinian (right noetherian, right perfect,
semilocal).
\end{proposition}

\begin{proof} By Lemma \ref{lema 1B},  we have the canonical projection
\begin{center}
$\pi:R*_{\alpha,w}G\rightarrow R*_{\alpha_{H}, w_H}H$.
\end{center}
So, we easily obtain that $R*_{\alpha_{H},w_H}H $ is right
artinian (right noetherian, right perfect, semilocal).
\end{proof}

\begin{remark}  If $H$ is the trivial subgroup $\{1_G\}$ of $G$, then
$R*_{\alpha_H, w_H}H=R$. So, by
Lemma \ref{lema 1B}, $R$ is a right (resp. left) $R$-direct
summand of $R*_{\alpha, w}G$ and the complement $A$ is given by
\begin{center}
$A=\{\sum a_gu_g\in R*_{\alpha}G \, |\, g\neq 1_G\}$.
\end{center}

\end{remark}



As an immediate consequence of  Proposition \ref{proposição 10} we have the
following result.

\begin{corollary}\label{corolario 4B}
Let $\alpha$ be an unital  twisted partial action of a group $G$ on a ring $R$. If
$R*_{\alpha,w}G$ is right artinian (right noetherian, right perfect,
semilocal), then $R$ is right artinian (right noetherian, right perfect,  semilocal).
\end{corollary}

To study necessary and sufficient conditions for the partial
crossed product to be right artinian, we need the  following three results, where we work with twisted global actions  where  we study artinianity in global crossed products. 

It is convenient to remind  that the
\emph{partial fixed ring}  is  $$R^{\alpha}=\{x\in R\, |\,
\alpha_g(x1_{g^{-1}})=x1_g, \forall\, g\in G\}.$$ See  \cite{SDF} and \cite{laz e mig} for more details.


\begin{lemma} \label{lema 13}
Let $\beta$ be a twisted global action of a group $G$ on a primitive ring $T$.
If $T*_{\beta,u}G$ is right artinian, then $G$ is finite.
\end{lemma}

\begin{proof}  By Corollary \ref{corolario 4B}  we have that  $T$ is right artinian.
Thus, $T$ is a simple artinian ring and we have that $Z(T)=F$ is a field. Let $F^{\beta}$
be the fixed field by $\beta$. Then, $F^{\beta}*_{\beta,u}G$ is right artinian since
$F^{\beta}*_{\beta,u}G$ is a direct summand of $T*_{\beta,u}G$. Note that
$F^{\beta}\subseteq K$, where $K$ is an algebraically closed field and it is
 not difficult to show that  we can extend $\beta$ to $F^{\beta}\otimes K$
  and consequently we have the isomorphisms $F^{\beta} *_{\beta,u} G\otimes K\simeq F^{\beta}\otimes K[G]\simeq K*_{\beta,u} G$
   that is right artinian, where $K*_{\beta,u} G$ is a twisted group ring. Hence, the twisted group ring $K*_{\beta,u}G$ is right artinian and it  has finite dimension as $K$-vector space. So,   $G$ is  finite.
\end{proof}

\begin{lemma} \label{lema 5B} Let $\beta$ be a twisted global  action of
a group $G$ on a semiprimitive ring $T$. If  $T*_{\beta,u}G$ is right artinian, then $T$ is right
artinian and $G$ is finite.
\end{lemma}

\begin{proof}  Suppose that  $T*_{\beta,u}G$ is artinian. Then by Corollary \ref{corolario 4B} we have that $T$ is right artinian and it follows that $T$ is semiprimitive Artinian.  Now, let $T=T_1\oplus\cdots\oplus T_n$ be the decomposition in simple components and  consider the subgroup $H=\{g\in G:\beta_g=id_T\}$ of $G$.   Note that the twisted group ring  $T_i*_ {\beta,u}H$  is right Artinian, for all $i\in \{1,\cdots,n\}$.  Now, as in the Lemma \ref{lema 13} we have that $H$ is finite. Hence,  using the similar methods of the second paragraph in the proof of  (\cite{park1}, Lemma 3.2) we get that $G$ is finite.  

The converse is trivial.  \end{proof}


Now, with the last two lemmas in mind, the proof of the next proposition follows with similar methods of (\cite{park1}, Theorem 3.3).

\begin{proposition}\label{pro11} Let $\beta$ be a twisted global action on a ring $T$. Then $T*_{\beta,u}G$ is right artinian if and only if $T$ is right artinian and $G$ is finite.\end{proposition}

Now, we are in conditions to prove the first main result of this section that generalizes (\cite{park1}, Theorem 3.3).

\begin{theorem} \label{teorema 6B}
Let $\alpha$ be an unital  twisted partial action of a group $G$ on a ring $R$ such that $\alpha$ is of finite type. Then  $R*_{\alpha,w}G$ is right artinian if and only if
 $R$ is right artinian and $G$ is a finite group.
\end{theorem}

\begin{proof}

 By Corollary \ref{corolario 4B} we have that $R$ is right artinian.  Thus, $R/J(R)$ is semiprimitive artinian and it is semisimple by (\cite{lam1}, Proposition 11.7). Since $R/J(R)*_{\overline{\alpha},\overline{w}}G$ is right artinian, then  we may assume that $R$ is a semisimple. In this case, 
$\alpha$ admits a globalization $(T, \beta, u)$ and following the same steps of (\cite{laz e
mig}, Corollary 1.3 and 1.8), $T$ is a semiprimitive right
artinian ring with identity. By Proposition \ref{morita}  we have that  $R*_{\alpha,w}G$ and
$T*_{\beta,u}G$ are Morita equivalent and it follows from (\cite{mcconnel robson}, Lemma 3.5.8) that $T*_{\beta,u}G$
is right artinian. So, by Proposition \ref{pro11} we have that  $G$ is
finite

The converse is trivial.
\end{proof}

Now, we present an example where the assumption ``$\alpha$ is of finite type"  on  Theorem \ref{teorema 6B} is not superfluous.

\begin{example} 
	Let $R=Ke_1\oplus Ke_2$, where $\{e_1,e_2\}$ is a set of central orthogonal idempotents and $K$ is a field.  We define the action of $G=\mathbb{Z}$ as follows:  $D_0=R$, $D_{-1}=Ke_{1}$, $D_{1}=Ke_2$,  $D_i=0$, for $i\neq 0,1,-1$, $\alpha_0=id_R$, $\alpha_{1}(e_1)=e_2$, $\alpha_{-1}(e_2)=e_1$ and $\alpha_i=0$ for $i\neq 0,1,-1$. We clearly have that  $R*_{\alpha}G$  is a finite dimensional  vector space  and we obtain that $R*_{\alpha}G$ is right artinian with $G$ an infinite group and $\alpha$ is not of finite type. 
\end{example}

One may ask if we input the assumption $D_g\neq 0$, $g\in G$ and $\alpha$ is any twisted partial action would be enough to get the result of Theorem 15. The next example shows that this is not the case.

\begin{example} 
	Let $R=Ke_1\oplus Ke_2$ and a  partial action  of $\mathbb{Z}$ of $R$ is defined as follows: $D_e=R$,  $\alpha_e=id_R$, $D_i=Ke_1$ and $\alpha_i=id_{Ke_1}$ for all $i\neq 0$. Then $R*_{\alpha}G=R\oplus (\oplus_{i\neq 0} Ke_1\delta_i$). Note that in this case $R*_{\alpha}G$ is commutative and we only have the ideals $I_1=Ke_1\oplus (\oplus_{i\neq 0} Ke_1\delta_i)$ and $I_2=Ke_2\oplus (\oplus_{i\neq 0} Ke_1\delta_i)$. Thus, $R*_{\alpha}G$ is artinian and $G$ is infinite.
\end{example} 



In \cite{CCF}, the authors studied the noetherianity of partial skew groups rings with the assumption that the group is polycyclic-by-finite.  In the next result we study the noetherianity of partial crossed products  where the group is not necessarily polycyclic-by-finite.

\begin{theorem}\label{teorema 7B}
Let $\alpha$  be an unital  twisted partial action of a group  $G$ on $R$
such that only a finite number of ideals $D_g$ are not  the zero
ideal. Then $R$ is right noetherian if and only if $R*_{\alpha,w
}G$ is right noetherian.
\end{theorem}

\begin{proof}
Suppose that $R$ is right noetherian. Then each ideal $D_{g}$, $g\in G$, is
finite generated (as a right $R$-module) and since $D_{g}$
and $\delta_gD_{g}$ are $R$-isomorphic, we have that each
$u_gD_{g}$ is finitely generated as a right $R$-module. Now, using the assumption
we have that  $R*_{\alpha,w}G=\bigoplus_{g\in
G}\delta_gD_{g}$ is finitely generated as a right $R$-module.
So, by (\cite{mcconnel robson}, Lemma 1.1.3),
$R*_{\alpha,w}G$ is right noetherian.
The converse follows from Corollary \ref{corolario 4B}.
\end{proof}

As a particular case of Theorem \ref{teorema 7B} we have the following result.

\begin{corollary}
Let $\alpha$ be an unital  twisted partial  action of a finite group $G$ on $R$.
Then $R$ is right noetherian if and only if $R*_{\alpha,w}G$ is
right noetherian.
\end{corollary}

Acoording to (\cite{reid}, pg. 1) the class of groups $\mathcal{H}$ is defined by the groups that are either finite or contains an infinite abelian subgroup. Now, we are in conditions to prove the following result that generalizes (\cite{park1}, Theorem 3.7)

\begin{theorem} \label{teorema 17} 

Let $\alpha$ be an unital  twisted partial action of  a group $G\in \mathcal{H}$ on $R$ such that $\alpha$ is of finite type.  Then  $R*_{\alpha,w}G$ is right perfect if and only if $R$ is right perfect and $G$ is finite. \end{theorem}

\begin{proof}  Suppose that $R*_{\alpha,w}G$ is right perfect.   By the fact that $\alpha$ is of finite type, we have by Proposition \ref{morita}  that $R*_{\alpha,w}G$ and $T*_{\beta,u}G$ are Morita equivalent rings then  we get that  that  $T*_{\beta,u}G$ is right perfect by (\cite{AF}, Corollary 28.6). Note that, by Corollary \ref{corolario 4B} and (\cite{lam1}, Corollary 24.19), $T$  and \linebreak $(T/J(T))*_{\bar{\beta},\bar{u}}G\cong (T*_{\beta, u}G)/(J(T)*_{\beta,u}G)$ are right perfect. Thus,
we may assume that $T$ is semisimple  and let $H=\{h\in G: \beta_h(x)=x, \forall\, x\in T\}$. Then $H$ is a subgroup of $G$ and by Proposition \ref{proposição 10}  we have that $T*_{\beta,u}H$ is a perfect  twisted group ring. Hence,  by (\cite{reid}, Theorem 2.6)  we have that $H$ is finite. Now, following the same arguments applied to the artinian case in (\cite{park1}, Lemma 3.2)  we get that $G/H$ is finite. So, $G$ is finite.




Conversely, assume that $R$ is right perfect and $G$ is finite. The proof follows the same arguments of (\cite{park1}, Theorem 3.7).
\end{proof}

If $\alpha$ is not of finite type,  the result of the last theorem is not true as the next example shows.

\begin{example}  Let $R=K$ and $G=\mathbb{Z}$. We define the following partial action: $D_0=R$ and $\alpha_0=id_K$, $D_s=0$ and $\alpha_s=0$ for all $s\neq 0$. Thus, $R*_{\alpha}G=R$ that is right perfect, with $G$ an infinite group, and $\alpha$ is not of finite type.
\end{example}  

In the next result, we study the perfect, semiprimary and semilocal property between a ring $R$ with a twisted partial action $\alpha$ and $(T,\beta, u)$ its enveloping action when it exists.

\begin{proposition} \label{proposicao 1D}
Let $\alpha$ be an unital  twisted partial action of a group $G$ on
a ring $R$ such that  $\alpha$ has an enveloping action  $(T, \beta, u)$. The following statements hold.

(i)  If $R$ is right perfect and $\alpha$ is of finite type, then $T$ is right perfect.

(ii) If $T$ is right perfect, then $R$ is right perfect.

(iii) If $R$ is semilocal (semiprimary)  and $\alpha$ is of finite type, then $T$ is semilocal (semiprimary).

iv) If $T$ is semilocal (semiprimary) then $R$ is semilocal (semiprimary).

\end{proposition}

\begin{proof}

$(i)$ Suppose that $R$ is right perfect and consider $Tt_1\supseteq Tt_2\supseteq \cdots$ a descending chain of
principal left ideals of $T$.  Since $\alpha$ if of finite
type then there are $g_1, \cdots , g_n \in G$ such that
$T=\beta_{g_1}(R) + \cdots +\beta_{g_n}(R)$. Thus, for each $i=1,2, \cdots n$, we
have
\begin{equation*}\label{DC}
\beta_{g_i}(1_R)Tt_1\supseteq \beta_{g_i}(1_R)Tt_2\supseteq \cdots
\end{equation*}
Now for each $i=1, 2, \cdots, n$ and $j\geq 1$,
\begin{center}
$\begin{array}{llll}
      x \in \beta_{g_i}(1_R)Tt_j  & \Leftrightarrow & x=\beta_{g_i}(1_R)tt_j, \,\ t\in T \\
      &\Leftrightarrow & x=t\beta_{g_i}(1_R)t_j, \,\ t\in T \\
       & \Leftrightarrow & x=t\beta_{g_i}(1_R)[\beta_{g_i}(1_R)t_j], \,\ t\in T \\
       & \Leftrightarrow & x=y\beta_{g_i}(1_R)t_j, \,\ y\in \beta_{g_i}(R) \\
       & \Leftrightarrow &  x\in \beta_{g_i}(R)[\beta_{g_i}(1_R)t_j]\\
    \end{array}$
\end{center}
Hence, for each $i=1, 2,\cdots, n$ and $j\geq 1$,
$\beta_{g_i}(1_R)Tt_j$ is an principal left ideal of
$\beta_{g_i}(R)$. Since each $\beta_{g_i}(R)$ is right perfect,
$i=1, 2, \cdots, n$, there exist $m_i$ such that
\begin{center}
$\beta_{g_i}(1_R)Tt_{m_i}=\beta_{g_i}(1_R)Tt_{m_{i+1}}= \cdots $.
\end{center}
Let $m=max\{m_i \, |\, 1\leq i \leq n\}$. Then we obtain that
\begin{center}
$\beta_{g_i}(1_R)Tt_m=\beta_{g_i}(1_R)Tt_{m+1}= \cdots $,
\end{center}
for each $i=1, 2, \cdots, n$. Following ([\cite{laz e mig}, Observation
1.13(iv)), we obtain $Tt_m=Tt_{m+1}=\cdots$. Finally, by the fact that $\alpha$ is of finite type we have that the extension of $\alpha$ to $R/J(R)$ is of finite type. So, using the same methods of  (\cite{laz e mig}, Corollary 1.3)  we have that $T/J(T)$ is artinian.  Therefore,  $T$ is perfect.

$(ii)$  Suppose that $T$ is right perfect and consider
\begin{center}
$Ra_1\supseteq Ra_2\supseteq \cdots $
\end{center}
a chain of principal left ideals of $R$. Following (\cite{laz e mig},
Observation 1.13 (ii)), this is also a chain of principal left
ideals of $T$. Since $T$ is right perfect, there exists $m$ such
that $Ra_m=Ra_{m+1}=\cdots$. Using the same arguments presented in  \cite{laz e mig},  we have that $(T/J(T), \overline{\beta}, \overline{u})$ is the enveloping action  of $(R/J(R), \overline{\alpha}, \overline{w})$, where $\overline{\beta}$ and $\overline{\alpha}$ are the extensions of $\beta$ and $\alpha$ to $T/J(T)$ and $R/J(R)$, respectively.  Since $T/J(T)$ is artinian, then   by  the same methods of (\cite{laz e mig}, Proposition 1.2 and Corollary 1.3) we get that $R/J(R)$ is right artinian.
So, $R$ is right perfect.

(iii) We only prove for the semilocal case, because the semiprimary case is similar. By assumption, $R/J(R)$ is semisimple artinian, and using the methods presented in \cite{laz e mig} we have that ($T/J(T), \overline{\beta}, \overline{u}$) is the enveloping action of ($R/J(R), \overline{\alpha}, \overline{w}$). 
Since $R/J(R)$ is artinian and the extension $\overline{\alpha}$ of $\alpha$ to $R/J(R)$ is of finite type,  then by (\cite{laz e mig}, Corollary 1.3) we have that $T/J(T)$ is artinian. Hence, $T/J(T)$ is semisimple artinian. So, $T$ is semilocal.

(iv) As in item $(iii)$, we only prove for the semilocal case. As mentioned in item (iii) we have that ($T/J(T), \overline{\beta}, \overline{u}$) is the enveloping action of ($R/J(R), \overline{\alpha}, \overline{w}$). Since  $R/J(R)$ is an ideal of $T/J(T)$ then $R/J(R)$ is artinian. By the fact that $R/J(R)$ is semiprimitive we have that $R/J(R)$ is semisimple. So, $R$ is semilocal.
\end{proof}

In general $R$ is right perfect  does not imply that  $T$ is right perfect as the next example shows which implies that the assumption of ``$\alpha$ is of finite type" in item $(i)$ above can not be dropped.

\begin{example} Let $T=\oplus_{i\in \mathbb{Z}}Ke_i$, where $\{e_i:i\in \mathbb{Z}\}$ is a set of central orthogonal idempotents, $K$ is a field, and $R=Ke_0$. We define an action of an infinite cyclic group $G$ generated by $\sigma$ as follows: $\sigma(e_i)=e_{i+1}$, for all $i\in \mathbb{Z}$. Thus we easily have that an induced  partial action of  $G$ on $R$. Note that $R$ is right perfect, but $T$ is not right perfect.
\end{example}



\begin{theorem} \label{teorema 20}
Let $\alpha$ be an unital twisted partial action of a group $G$ on a ring $R$ and $(T,\beta,u)$ its enveloping action. Suppose that $G\in \mathcal{H}$ and $\alpha$ is of finite type. Then the following conditions are equivalent:
\begin{itemize}
\item[(i)] $R$ is right perfect and $G$ is finite.
\item[(ii)] $T$ is right perfect and $G$ is finite.
\item[(iii)] $T*_{\beta,u}G$ is right perfect.
\item[(iv)] $R*_{\alpha,w}G$ is right perfect.
\end{itemize}
\end{theorem}

\begin{proof} 
By Proposition \ref{proposicao 1D} we have that $(i)$ and $(ii)$ are equivalent.
By the fact that $\alpha$ is of finite type we have by  Proposition \ref{morita}  that  $T*_{\beta,u}G$ and $R*_{\alpha,w}G$ are Morita equivalent and using  ( \cite{AF}, Corollary 28.6) we easily have that $(iii)$ and $(iv)$ are equivalent. Moreover, by Theorem \ref{teorema 17}, we have that $(i)$ and $(iv)$ are  equivalent.
\end{proof}

\begin{remark} 
It is convenient to point out that Theorem \ref{teorema 20} holds without the assumption that $G\in \mathcal{H}$ if the twisted partial action is simply  a partial action 
\end{remark}

In the next two results we study necessary and sufficient conditions for the partial crossed product to be semiprimary and  semilocal with the assumption that the group is in the class 
$\mathcal{H}$ and these generalizes (\cite{park1}, Theorem 3.8),  (\cite{park1}, Proposition 4.2) and (\cite{jan}, Corollary 2).

\begin{theorem} \label{proposicao 4D}
Let $\alpha$ be an unital  twisted partial action of a group $G$ on a ring $R$  such that $G\in \mathcal{H}$. Then
 $R*_{\alpha,w}G$ is semiprimary if and only if $R$ is
semiprimary and $G$ is a finite group.
\end{theorem}
\begin{proof}
Suppose that  $R*_{\alpha,w}G$ is
semiprimary. Then $R*_{\alpha,w}G$ is right perfect and by Theorem \ref{teorema 17}, $G$ is finite and $R$ is right perfect.
Hence, $R/J(R)$ is semisimple. Since $J(R*_{\alpha}G)$ is nilpotent (because
$R*_{\alpha, w}G$ is semiprimary), then following (\cite{laz e mig},
Corollary 6.5(ii)), we have that  $J(R)$ is nilpotent. So, $R$ is
semiprimary.

Conversely,  suppose that $R$ is semiprimary and $G$ is finite. Then
$\bar{R}=R/J(R)$ is semisimple and, consequently, it is right
artinian. Hence following (\cite{laz e mig}, Corollary 6.5(ii)),
$J(R*_{\alpha,w}G)$ is nilpotent.  Note that by Theorem \ref{teorema 17}, we have that $R*_{\alpha,w}G$ is right perfect.  So, $R*_{\alpha,w}G$ is semilocal and we have that $R*_{\alpha,w}G$ is semiprimary.
\end{proof}

\begin{theorem}\label{proposicao 5D}
Let $\alpha$ be an unital  twisted  partial action of a finite group $G$ on a ring
$R$. Then $R$ is semilocal
if and only if $R*_{\alpha,w}G$ is semilocal.
\end{theorem}

\begin{proof}
If  $R*_{\alpha,w}G$ is semilocal, then by Corollary \ref{corolario 4B}, $R$ is semilocal.
Conversely, suppose that $R$ is
semilocal. Then $\bar{R}=R/J(R)$ is semisimple. Hence $R/J(R)$ is right artinian  and using  Theorem \ref{teorema 6B} and (\cite{CM}, Lemma 3.8) we get that $(R/J(R))*_{\alpha,w}G$ is semisimple. Thus,
$(R*_{\alpha,w}G)/J(R*_{\alpha,w}G)$ is semisimple. So,
$(R*_{\alpha,w}G)$ is semilocal.
\end{proof}

Now, we study the semilocal property for partial crossed products when the group is not necessarily in the class $\mathcal{H}$.



Let $G$ be a group and $\alpha =
\big(\{D_g\}_{g \in G}, \{\alpha_g\}_{g \in G}, \{w_{g,h}\}_{(g,h)
\in G\times G}\big)$ a twisted partial action of a group $G$ on a semiprime ring $R$. From \cite{lmsw}, $\alpha$ can be extented to a twisted partial action $\alpha^* =
\big(\{D_g^*\}_{g \in G}, \{\alpha_g^*\}_{g \in G}, \{w_{g,h}^*\}_{(g,h)
\in G\times G}\big)$ of $G$ on the left (or right) Martindale ring of quotients $Q$ of $R$. If each $D_g$ is generated by a central idempotent $1_g\in R$, then $D_g^*$ is also generated by $1_g$.  Following \cite{SDF}, we say that $\alpha$ is $X$-\textit{outer}, if for all $g\in G-\{1_G\}$, the set $\phi_g=\{x\in D_g^*\, |\, x\alpha_g(a1_{g^{-1}})=ax, \,\ for \,\ all \,\ a\in R\}$ is zero. The following proposition is proved in (\cite{SDF}, Corollary 4.10).

\begin{proposition} \label{saradia 4.10}
Let $\alpha$ be an unital  twisted partial action of a group $G$ on a semiprimitive ring $R$. If $\alpha$ is $X$-outer, then $R*_{\alpha,w}G$ is semiprimitive 
\end{proposition}

The next result  generalizes (\cite{park1}, Proposition 4.4).

\begin{proposition} 
Let $\alpha$ be an unital  twisted partial action of a group $G$ on a semiprimitive ring $R$. If $R*_{\alpha, w}G$ is semilocal and $\alpha$ is partial outer, then $G$ is finite.
\end{proposition}

\begin{proof} 
By Proposition \ref{saradia 4.10},  we have that $R*_{\alpha,w}G$ is semiprimitive. Then $R*_{\alpha,w}G$ is artinian and by Theorem \ref{teorema 6B}, we have that $G$ is finite. 
\end{proof}

We finish this section with the following  two results  where we only consider  partial actions. The next result  generalizes (\cite{jan}, Theorem 2 )

\begin{corollary}
Let $\mathcal{K}$ be a field of characteristic zero, $R$ a
$\mathcal{K}$-algebra and $\alpha$ is partial action of a group $G$ on $R$. Suppose that  $R$  admits an enveloping action $(T,\beta)$ and $\alpha$ is of finite type. Then  $R*_{\alpha}G$ is semilocal if and only if $R$ is semilocal and $G$ is finite.
\end{corollary}

\begin{proof}
Since $\alpha$ is of finite type, then by Proposition \ref{morita}   we have that $R*_{\alpha,}G$  and $T*_{\beta}G$  are Morita equivalent. Thus, by the methods presented in \linebreak (\cite{AF} Chapter 28)  we have that $T*_{\beta}G$ is semilocal.  Hence, $G$ is finite by \linebreak (\cite{jan},
Theorem 2) and $R$ is semilocal by Corollary \ref{corolario 4B}.
\end{proof}

A group $G$ is said to be \textit{torsion}  if all its elements are of finite order.  We finish this section with the following result that generalizes \linebreak (\cite{jan}, Corollary 1).

\begin{proposition}
Let $\alpha$ be a partial action of finite type of a group $G$ on
a ring $R$ that has enveloping action $(T,\beta)$. If
$R*_{\alpha}G$ is semilocal then $R$ is semilocal and $G$ is
torsion.
\end{proposition}

\begin{proof}
If $R*_{\alpha}G$ is semilocal, then $R$ is semilocal by Corollary
\ref{corolario 4B}. By Proposition \ref{morita} we have that  $R*_{\alpha}G$ and
$T*_{\beta}G$  are Morita equivalent and by the methods presented in \cite{AF} we get that $T*_{\alpha}G$ is semilocal.  So, by [(\cite{jan}, Corollary
1), $G$ is torsion.
\end{proof}

\section{Krull Dimension}\

In the previous section we studied the artinianity of the partial
crossed product, and now we will study the Krull dimension, which
is a measure for a module or a ring to be  artinian.

We recall that the \emph{Krull dimension} of a right $R$-module
$M$, which we will be denoted by $Kdim(M_R)$, is defined to be the
deviation of the lattice $\mathcal{L}_R(M)$ of $R$-submodules of $
M$. The \emph{right Krull dimension} of a ring $R$ is defined to
be the Krull dimension of the right $R$-module $R$, which we shall
denote by $Kdim(R)$ if there is no possibility of
misunderstanding. We note that the right modules of Krull
dimension $0$ are precisely non-zero right artinian modules. Also
is well-known that every right noetherian modules has Krull
dimension. For more details and basic properties of Krull
dimension, see \cite{mcconnel robson}.

During this section, unless otherwise stated, $\alpha$ is only a partial action of a group $G$ on a ring $R$ such that the ideals $D_g$, $g\in G$, are generated by central idempotents $1_g\in R$. 

 In the next two results we get some properties of the group when the partial skew group ring has Krull dimension that generalizes (\cite{park2}, Proposition 1) and (\cite{park2}, Lemma 6), respectively.

\begin{proposition} \label{park prop 1} Let $\alpha$ be a partial action of a group $G$ on $R$.  If $R*_{\alpha}G$ has Krull dimension, then $G$ satisfies ACC  on finite subgroups.\end{proposition}

\begin{proof}  It follows from the same methods of ( \cite{park2}, Proposition 1) \end{proof}

\begin{proposition} Let $\alpha$ be a partial of a group $G$ on $R$. If $R*_{\alpha}G$ has Krull dimension, then $G$ satisfies ACC on normal subgroups\end{proposition}

\begin{proof} Using similar methods of  (\cite{park2}, Lemma 6) we have the result.\end{proof}

Now, we get a comparison of the Krull dimension between the $R*_{\alpha}G$ and $R$.

\begin{proposition}\label{teorema 1C} 
Let $\alpha$  a  partial action of a finite group $G$ on a ring
$R$. If $R$ has right Krull dimension, then $R*_{\alpha}G$ has
right Krull dimension and $Kdim(R*_{\alpha}G)\leq Kdim(R)$.
\end{proposition}

\begin{proof}
If $R$ has right Krull dimension, then each ideal $D_g$ has Krull
dimension as a right $R$-module and this implies that each
$\delta_gD_{g}$ has Krull dimension as a right $R$-module. Thus,
$Kdim(D_{g})_R=Kdim(\delta_gD_{g})_R$. Hence, by
(\cite{mcconnel robson}, Lemma 6.1.14),
$R*_{\alpha}G=\bigoplus_{g\in G}\delta_gD_{g}$ has Krull dimension
as a right $R$-module and
\begin{eqnarray*}
Kdim\,(R*_{\alpha}G)_R&=&sup\, \{Kdim \,((\delta_gD_{g})_R)\, | \, g\in G\}\\
&=&sup\, \{Kdim ((D_{g})_R)\, | \, g\in G\}\\
&=&sup\, \{Kdim(R), \,Kdim ((D_{g})_R)\, | \, g\in G-\{1_G\}\}\\
&=&Kdim(R), \,\ \emph{because of} \,\ Kdim ((D_{g})_R)\leq
Kdim(R).
\end{eqnarray*}

Note  that the map $\varphi:
\mathcal{L}_{R*_{\alpha}G}(R*_{\alpha}G) \longrightarrow
\mathcal{L}_{R}(R*_{\alpha}G)$ defined by $\varphi(N)=N$ is
strictly increasing and then, by (\cite{mcconnel robson},
Proposition 6.1.17), $R*_{\alpha}G$ has Krull dimension and
$Kdim(R*_{\alpha}G)\leq Kdim(R*_{\alpha}G)_R$.

So, $Kdim(R*_{\alpha}G)\leq
Kdim\,((R*_{\alpha}G)_R)=Kdim(R)$.
\end{proof}

Using similar methods of  (\cite{park2}, Theorem 2.4) we get the following.

\begin{corollary} 
	Let $\alpha$ be a  partial action of an abelian  group $G$ on $R$. If $R*_{\alpha}G$ has Krull dimension then $G$ is finitely generated. 
\end{corollary}

We recall that a ring $S$ is a \emph{right Goldie ring}, if $S$
has $ACC$ on right annihilator ideals and $S$ has finite right
uniform dimension. Also by \linebreak (\cite{mcconnel robson}, Proposition
6.3.5), any semiprime ring with right Krull dimension is a right
Goldie ring.

\begin{corollary}\label{corolario 2C}
Let $\alpha$ be a partial action of any group $G$ on a
semiprime ring $R$ . If $R$ has right Krull
dimension then $R*_{\alpha}G$ is a right Goldie ring.
\end{corollary}

\begin{proof}
If $R$ is a semiprime ring with Krull dimension, then $R$ is a
semiprime right Goldie ring. Hence, by (\cite{CM}, Lemma 3.6), $R*_{\alpha}G$ is semiprime and by Proposition \ref{teorema
1C}, $R*_{\alpha}G$ has Krull dimension. So, $R*_{\alpha}G$
is a Goldie ring.
\end{proof}

Now, using Krull dimension, we study the artinianity between $R$ and $R*_{\alpha,w}G$.

\begin{corollary} \label{corolario 3C}
Let $\alpha$ be  a  partial action of a group $G$ on a ring $R$ where only a finite number of the ideals $D_g$ is not zero.
Then,  $R*_{\alpha}G$ is right
artinian if and only if $R$ is right artinian and $G$ is finite.
\end{corollary}

\begin{proof}
It is a direct consequence of Theorem \ref{teorema 6B} and
Proposition \ref{teorema 1C}, since  right artinian
rings are precisely the  rings with right Krull dimension $0$.
\end{proof}

\begin{proposition}\label{proposicao 4C}
Let $\alpha$ be a  partial action of a group $G$ on a ring $R$ and $H$ a subgroup of $G$. If $R*_{\alpha}G$ has right Krull dimension,
then $R*_{\alpha_H}H$ has right Krull dimension and
$Kdim(R*_{\alpha}G)\leq Kdim(R*_{\alpha_H}H)$. In particular, if
$R*_{\alpha}G$ has right Krull dimension, then $R$ has right Krull
dimension and $Kdim(R*_{\alpha}G)\geq Kdim(R)$.
\end{proposition}

\begin{proof}
By Proposition \ref{proposicao 2.1 park1} it is easy to see that
the map
\begin{center} $\varphi:
\mathcal{L}_{R*_{\alpha_H}H}(R*_{\alpha_H}H) \longrightarrow
\mathcal{L}_{R*_{\alpha}G}(R*_{\alpha}G)$
\end{center}
defined by $\varphi(I)=I(R*_{\alpha}G)$ is strictly increasing.
So,
 by (\cite{mcconnel robson}, Proposition 6.1.17) we have that $R*_{\alpha_H}H$
has Krull dimension and  
\begin{center} 
	$Kdim(R*_{\alpha_H}H)\leq
Kdim(R*_{\alpha}G)$.
\end{center}
\end{proof}

An immediate consequence of Propositions \ref{teorema 1C} and
\ref{proposicao 4C} we have the  following result.

\begin{corollary}\label{corolario 4C}
Let $\alpha$ be a partial action of finite group $G$ on a ring
$R$. Then $R$ has right Krull dimension if and only if
$R*_{\alpha}G$ has right Krull dimension. In this case, $Kdim(R)=
Kdim(R*_{\alpha}G)$.
\end{corollary}

According to \cite{park2} a group $G$ is locally finite if each finitely generated subgroup of $G$ is finite.

Next we have a partial generalization of (\cite{park2}, Corollary
2).

\begin{theorem}\label{teorema 6C}
Let $\alpha$ be a partial action of a locally finite group $G$ on
a  ring $R$. Then
$R*_{\alpha}G$ has right Krull dimension if and only if $R$ has
right Krull dimension and $G$ is finite.
\end{theorem}

\begin{proof}
Suppose that $R*_{\alpha}G$ has
right Krull dimension. Then by Proposition \ref{proposicao 4C},
$R$ has Krull dimension. We claim that $G$ is finite.  In fact, suppose that $G$ is infinite. Thus, there exists $f_i\in G$, $i\geq 1$ such that $f_i\notin \langle f_1,\cdots,f_{i-1}\rangle$ and $\langle f_1\rangle \subsetneqq  \langle f_1,f_2\rangle \subsetneqq \cdots$ is strictly increasing which contradicts the Proposition \ref{park prop 1}. 
The converse follows from Proposition \ref{teorema 1C}.
\end{proof}



\section{Triangular matrix representations of partial skew group rings}\

In this section, we characterize the enveloping action of partial actions of groups on triangular matrix algebras and we study triangular matrix representation of partial skew group rings. 

We begin with the following definition which partially generalizes (\cite{HG}, Definition 2.7).

\begin{definition} Let $N$ be a $(R,S)$-bimodule, $\alpha_1=\{\alpha_{R}^{g}:D_{R}^{g^{-1}}\rightarrow D_{R}^{g}: g\in G, D_{R}^{g} \triangleright R\}$, $\alpha_2=\{\alpha_{S}^g:D_{S}^{g^{-1}}\rightarrow D_S^{g}: g\in G, D_{S}^{g} \triangleright S\}$ be  partial  actions of $G$ on $R$ and $S$, respectively.  A partial action $\alpha$ of $G$ on $N$ relative to  $(\alpha_1, \alpha_2)$ if there exists a collection of $(D_{R}^{g}, D_{S}^{g})$-bi-submodules $N_g$ of $N$ and isomorphims of abelian groups $\alpha_g:N_{g^{-1}}\rightarrow N_g$ such that the following properties are satisfied:

(i) $\alpha_g(r_{g^{-1}}n_{g^{-1}})=\alpha_1^g(r_{g^{-1}})\alpha_g(n_{g^{-1}})$ and $\alpha_g(n_{g^{-1}}s_{g^{-1}})=\alpha_g(n_{g^{-1}})\alpha^{g}_2(s_{g^{-1}})$, for all $r_{g^{-1}}\in D^{g^{-1}}_R$, $s_{g^{-1}}\in D^{g^{-1}}_S$ and $n_{g^{-1}}\in N_{g^{-1}}$.

(ii) $\alpha_e=id_N$ and $N_e=N$.

(iii) $\alpha_g(N_{g^{-1}}\cap N_h)=N_g\cap N_{gh}$, for all $g,h\in G$.

(iii) $\alpha_g\circ \alpha_h(x)=\alpha_{gh}(x)$, for all $x\in \alpha_{h^{-1}}(N_h\cap N_{g^{-1}})$. \end{definition}

In the next result, we study necessary conditions for the existence of enveloping actions of partial actions  on bimodules and the proof similarly follows as in   (\cite{DE}, Theorem 4.5) and we put it here for reader's convenience.

\begin{theorem} Let $N$ be a $(R, S)$-bimodule, $\alpha$  a relative $(\alpha_1, \alpha_2)$-partial action of $G$ on $N$, where $\alpha_1$ is a partial action of $G$ on $R$ and $\alpha_2$ is a partial action of $G$ on $S$. If $(R,\alpha_1, G)$ and $(S, \alpha_2, G)$ have enveloping actions $(T_1, \gamma, G)$ and $(T_2, \theta,G)$, respectively,  such that  $1_g^Rn=n1_g^S=n$, for all $n\in N$, then there exists a $(T_1, T_2)$-bimodule $M$ with a global action $\beta$ of $G$ on $M$ such that the following properties are satisfied:

(i) $\beta_g(zm)=\gamma_g(z)\beta_g(m)$ and $\beta_g(ms)=\beta_g(m)\theta_g(s)$, for all $m\in M$, $g\in G$, $z\in T_1$ and $s\in T_2$.

(ii) $N$ is a $(T_1,T_2)$-submodule of $M$.

(iii) $\beta_g|_{N_{g^{-1}}}=\alpha_g$, for all $g\in G$.

(iv) $\beta_g(N)\cap N=N_g$, for all $g\in G$.\end{theorem}

\begin{proof}  We consider $F(G,N)$, $F(G,R)$ and $F(G,S)$ be the Cartesian product of the copies of $N$, $R$ and $S$  indexed by the elements of G, respectively,  that is, the algebra of all functions of $G$ into $N$, $G$ into $R$ and $G$ into $S$.   We define the actions $(fg)(s)=f(s)g(s)$ and $(hf)(s)=h(s)f(s)$, for all $s\in G$,  $f\in F(G,N)$, $g\in F(G,S)$ and $h\in F(G,R)$. Note that  $F(G,N)$ is a $(F(G,R),F(G,S))$-bimodule. As in the proof of (\cite{DE}, Theorem 4.5)   we have global actions $\beta_1$ and $\beta_2$ of $G$ on $F(G,R)$ and 
$F(G,S)$, respectively.  We define the relative $(\beta_1, \beta_2)$-global action $\beta$ of $G$ on $F(G,N)$ by $\beta_g(f)|_{h}=f(h^{-1}g)$. It is not difficult to show that $\beta$ is an isomorphim. Next, we define the homomorphism of $(R,S)$-bimodules $\varphi:N\rightarrow F(G,N)$  by $\varphi(n)|g=\alpha_g(1^{R}_{g^{-1}}n)=\alpha_g(n1_{g^{-1}}^{S})$. We clearly have that $\varphi$ is injective. Let $M=\sum_{g\in G}\beta_g(\varphi(N))$. Then  for each $g\in G$, we have that $\beta_g(N)$ is a $(\gamma_{g}(R),\theta_{g}(S))$-bimodule with the actions $\gamma_{g}(r)\beta_g(n)=\beta_g(rn)$ and $\beta_g(n)\gamma_g(s)=\beta_g(ns)$, for all $g \in G$, $r\in R$, $s\in S$ and $n\in N$.

Since, for all $g\in G$, $\gamma_g(R)$ and $\theta_g(S)$ are ideals of $T_1$ and $T_2$, respectively,  then we easily obtain that  $M$ is a $(T_1,T_2)$-bimodule. The other requirements of the theorem follows by the same methods of (\cite{DE}, Theorem 4.5). \end{proof}

Now, we recall the well-known definition of a triangular matrix algebra associated to a $(R,S)$-bimodule $N$.

\begin{definition} Let $N$ be an $(R,S)$-bimodule. The triangular matrix algebra associated to the triple $(R,N,S)$ is the algebra $\mathcal{L}=\{(r,n,s):r\in R, n\in N, s\in S \}$ with usual sum and multiplication rule  $(r,n,s).(r_1,n_1,s_1)=(rr_1,rn_1+ns_1, ss_1)$. \end{definition} 

\begin{remark} It is convenient to point out that \begin{center}
$\mathcal{L}=\left\{\left(\begin{array}{cc}
r & n \\ 
0 & s
\end{array}\right): r\in R, n\in N, s\in S\right\}$
\end{center}
\end{remark} 

From now on we denote the triangular matrix algebra associated to the triple $(R,N,S)$ as $\mathcal{L}=(R,N,S)$. 

The next result  is probably well-known, but we could not find a proper reference and we will give a sketch of the proof.

\begin{lemma} Let $\mathcal{L}=(R,N,S)$ be a triangular matrix algebra, where $N$ is $(R,S)$-bimodule. Then for any  ideal $J$ of $\mathcal{L}$ is of the form $J=(J_1, N_2,J_3)$, where $J_1$ is an ideal of $R$,  $J_3$ is an ideal of $S$   and $N_2$ is an $(R,S)$-bimodule. Moreover, if $J$ is generated by a central idempotent, then $J_1$ and $J_3$ are generated by central idempotents. \end{lemma}

\begin{proof} We define 
\begin{center}	
$J_1=\left\{a\in R| \exists \left(\begin{array}{cc}
a & 0 \\ 
0 & 0
\end{array}\right)\in J\right\}$,
\end{center}
 $$N_2=\left\{n\in N|\exists \left(\begin{array}{cc}
0 & n \\ 
0 & 0
\end{array} \right)\in J\right\},$$ 
$$J_3=\left\{c\in S|\exists \left(\begin{array}{cc}
0 & 0 \\ 
0 & c
\end{array}\right)\in J\right\}.$$ We clearly have that $J_1$ and $J_3$ are ideals of $R$ and $S$ and $N_2$ is an $(R,S)$-bimodule.  We easily have that $J=(J_1, N_2, J_3)$. Note that each idempotent in $\mathcal{L}$ is of the form $\left(\begin{array}{cc}
e & 0 \\ 
0 & f
\end{array}\right)$, where $e$ and $f$ are idempotents of $R$ and $S$, respectively.  The result follows. \end{proof}

The following result gives a characterization of the existence of enveloping actions of partial actions  of groups on triangular matrix algebras.

\begin{theorem} Let $\mathcal{L}=(R,N,S)$ where $R$ and $S$ are rings with identity, $N$  a $(R,S)$-bimodule   and $\alpha$ a partial action of a group $G$ on $\mathcal{L}$, where $\alpha=\{\alpha_g: T_{g^{-1}}\rightarrow T_g: g\in G\}$. Then $(\mathcal{L}, \alpha)$ has enveloping action $(T,\beta)$  such that for each $g\in G$ $\alpha_g(1_{g^{-1}}^R, 0,0)=(1_g^R,0,0)$ and $\alpha_g(0,0,1_{g^{-1}}^S)=
(0,0,1_{g}^{S})$ if and only if there exists partial actions $\alpha_1$ of $G$ on $R$, $\alpha_3$ of $G$ on $S$ and a relative partial $(\alpha_1,\alpha_3)$-action $\alpha_2$ of $G$ on $N$ such that $(R,\alpha_1)$, $(N,\alpha_2)$ and $(S,\alpha_3)$ have enveloping actions. Moreover, $T$ is a triangular matrix algebra.\end{theorem}

\begin{proof}  Suppose that $(\mathcal{L}, \alpha)$ has enveloping action $(T,\beta)$.  We only construct the partial action for the ring $R$, because the other constructions are similar. By  (\cite{DE}, Theorem 4.5) each ideal $T_g$, $g\in G$, is generated by a central idempotent  and  by Lemma 47, each ideal $T_g$ of $\mathcal{L}$ associated to the partial action $\alpha$ of $G$ is of the form $T_g=(R_g, N_g, S_g)$, where $N_g$ is a $(R,S)$-sub-bimodule of $N$ and $R_g$ and $S_g$  are generated by central idempotents  $1_g^R$ and $1_g^S$, respectively. For each $g\in G$,  we consider $A_g=\beta_g(R,0,0)$. We claim that $A_g\cap (R,0,0)=(R_g,0,0)$. In fact, let $y\in A_g\cap (R,0,0)$. Since $y=(a',b',c')$, then we obtain that $y=y(1_R,0,0)=(a',0,0)$. Thus, $y\in (R_g,0,0)$. 

On the other hand, let $(a,0,0)\in (R_g,0,0)$. Then, we have that $(a,0,0)=\beta_g(r,n,s)$. By the fact that , $\beta(1_{g^{-1}}^R,0,0)=(1_g^R,0,0)$  we have that   $(a,0,0)=(a,0,0)(1_g,0,0)=\beta_g(r,n,s)\beta_g(1_{g^{-1}},0,0)=\beta_g(r1_{g^{-1}},0,0)\in \beta_g(R,0,0)$.

 
 Note that   $\alpha_g(R_{g^{-1}},0,0)=\beta_g(R_{g^{-1}},0,0)=(R_g,0,0)$.  Thus, for each $g\in G$ we consider  the ideals $R_g$ and we define the isomorphisms $\alpha_g^R:R_ {g^{-1}}\rightarrow R_g$ by $\alpha_g^R=\pi_g\circ \alpha_g\circ i_g$, where $\pi_g:(R_g,0,0)\rightarrow R_g$ is the natural projection and $i_g:R_g\rightarrow (R_g,0,0)$ is the natural inclusion. Using the fact that $\alpha$ is a partial action of $G$ on $\mathcal{L}$ we  easily obtain  that $\alpha_1=\{\alpha_g^R:R_{g^{-1}}\rightarrow R_g\}$ is a partial action of $G$ on $R$.

Conversely,  by assumption and (\cite{DE}, Theorem 4.5) we have that $(R,\alpha_1)$ and  $(S,\alpha_3)$ have enveloping actions $(T_1,\beta_1)$ and $(T_3,\beta_3)$ and by Theorem 42 , $(N,\alpha_2)$ has enveloping action $(M,\beta_2)$. We define a global action of $G$ on $(T_1,M,T_3)$ as follows: for each $g\in G$, $\gamma_g(a,b,c)=(\beta_{1}^{g}(a),\beta_{2}^{g}(b),\beta_{3}^{g}(c))$. It is standard to show that  $Q=\sum_{g\in G} \gamma_g(R,N,S)$ with the action defined as above is an enveloping action for $(\mathcal{L},\alpha)$ and $Q$ is a triangular matrix algebra. \end{proof}
 
 \begin{lemma} Let $\theta:(R,N,S)\rightarrow (R',N',S')$ be an isomorphim of algebras  where $N$ is a $(R,S)$-bimodule,  $N'$ is a $(R',S')$-bimodule and,  $R$ and $S$  are not necessarily unital rings. The following conditions are equivalent:
 
 (a) $\theta(r,n,s)=(\theta_1(r),\theta_2(n),\theta_3(s))$, where $\theta_1:R\rightarrow R'$ and $\theta_3:S\rightarrow S'$ are isomorphims of rings and $\theta_2:N\rightarrow N'$ is a $(R,S)$-bimodule isomorphism.
 
 (b) $\theta(R,0,0)\subseteq (R',0,0)$ and $\theta(0,0,S)\subseteq (0,0,S')$.\end{lemma}
 
 \begin{proof}   
 The proof that $(a)$ implies $(b)$ is standard. For the converse, we define $\theta_1:R\rightarrow R'$, $\theta_2:N\rightarrow N'$ and $\theta_3:S\rightarrow S'$ by $(\theta_1(r),0,0)=\theta(r,0,0)$, $(0,\theta_2(n),0)=\theta(0,n,0)$ and $(0,0,\theta_3(s))=\theta(0,0,s)$.  It is standard to show that  $\theta_1$ and $\theta_3$ are homomorphism of  rings. Note that $(0,\theta_2(rn),0)=\theta(0,rn,0)=\theta((r,0,0)(0,n,0))=(\theta_1(r),0,0)(0,\theta_2(n),0)=(0,\theta_1(r)\theta_2(n),0)$. Thus, $\theta_2(rn)=\theta_1(r)\theta_2(n)$ and by similar methods we show that $\theta_2(ns)=\theta_2(n)\theta_3(s)$. Moreover, by the fact that $\theta$ is an isomorphism, we have that $\theta_1$,  $\theta_2$ and $\theta_3$ are isomorphisms.  \end{proof}




The following result under certain conditions, we characterize the partial skew group rings over a triangular matrix algebras, i.e, a triangular matrix representation of partial skew group rings.

\begin{theorem} Let $\mathcal{L}=(R,N,S)$ and $\alpha=\{\alpha_g:D_{g^{-1}}\rightarrow D_g, g\in G\}$ a partial action of $G$ on $\mathcal{L}$. Suppose that for each $g\in G$,  the ideals $D_g$ are generated by central idempotents.  Then there exists  partial actions $\alpha_1$, $\alpha_3$ of $G$ on $R$ and $S$, a relative $(\alpha_1,\alpha_3)$-partial action $\alpha_2$ of $G$ on $N$ such that $\mathcal{L}*_{\alpha}G\simeq (R*_{\alpha_1}G, M, S*_{\alpha_3}G)$, where $M$ is a $(R*_{\alpha_1}G, S*_{\alpha_3}G)$-bimodule. \end{theorem}

\begin{proof}  By assumption, Lemmas  47 and 49 and Theorem 48, for each $g\in G$, we have the  ideals $R_g$, $S_g$ of $R$ and $S$, submodules $N_g$ of $N$ with isomorphisms \begin{center} $\alpha^{g}_{1}:R_{g^{-1}}\rightarrow R_g$,\end{center} 
$\alpha^{g}_{2}:N_{g^{-1}}\rightarrow N_g$ and $\alpha^g_3:S_{g^{-1}}\rightarrow S_g$, and we have the  partial actions  $\alpha_1$ and $\alpha_3$  of $G$ on $R$ and $S$ and $\alpha_2$  a relative $(\alpha_1,\alpha_2)$-partial action of $G$ on $N$.   We claim that $\mathcal{L}*_{\alpha}G\simeq (R*_{\alpha_1}G,M,S*_{\alpha_3}G)$, where  $M$  is a   $(R*_{\alpha_1}G, S*_{\alpha_3}G)$-bimodule.  In fact, by Theorem 44 and (\cite{DE}, Theorem 4.5) we have that  $(\mathcal{L},\alpha)$,  $(R,\alpha_1)$, $(N,\alpha_2)$ and $(S,\alpha_3)$ have enveloping actions $(T,\beta)$, $(T_1,\beta^1)$, $(T_2,\beta^2)$ and $(T_3,\beta^3)$, respectively. By similar methods presented in \cite{HG} we have that $T*_{\beta}G=(T_1*_{\beta^1}G, T_2*_{\beta^2}G,T_3*_{\beta^3}G)$, where $T_2 *_{\beta^2}G$ is a $(T_1*_{\beta^1}G,T_3*_{\beta^3}G)$-bimodule whose elements are the finite sums $\sum_{g\in G} a_g\delta_g$ with usual sum and multiplication rule is $(s_h\delta_h)(a_g\delta_g)=s\beta^2_h(a_g)\delta_{hg}$ and $(a_g\delta_g)(w\delta_h)=a_g\beta^3_g(w)\delta_{gh}$, for all $s_h\delta_h\in T_1*_{\beta^1}G$, $a_g\delta_g\in T_2*_{\beta^2}G$ and 
$w\delta_h\in T_3*_{\beta^3}G$. We consider $M=\{f\in T_2*_{\beta_2}G: (0,f,0)\in \mathcal{L}*_{\alpha}G\}$
 and let  \begin{center} $y\in (R*_{\alpha_1}G,M,S*_{\alpha_3}G)$. \end{center}  Then $y=(\sum_{g\in G}r_g\delta_g, \sum_{g\in G}n_g\delta_g,\sum_{g\in G}s_g\delta_g)=(\sum_{g\in G}r_g\delta_g,0,0)+(0,\sum_{g\in G}n_g\delta_g,0)+(0,0,\sum_{g\in G}s_g\delta_g)=\sum_{g\in G}(r_g,0,0)\delta_g+\sum_{g\in G}(0,n_g,0)\delta_g+\sum_{g\in G}(0,0,s_g)\delta_g\in \mathcal{L}*_{\alpha}G$.

On the other hand, for each $(r_g,s_g,n_g)\delta_g\in \mathcal{L}*_{\alpha}G$ we have that $(r_g,n_g,s_g)\delta_g=(r_g\delta_g, n_g\delta_g,s_g\delta_g)\in (R*_{\alpha_1}G, M,S*_{\alpha_3}G)$, since $n_g\delta_g\in M$ because of $n_g\delta_g\in T_1*_{\beta_1}G$ and $(0,n_g\delta_g,0)\in \mathcal{L}*_{\alpha}G$.  So, the result follows.\end{proof}

As an immediate consequence we have the following result.

\begin{corollary} Let $\alpha$ be a partial action of a group $G$ on $R$. Then the partial action $\alpha$ extends to a partial action $\overline{\alpha}$ on $\mathcal{L}=(R,R,R)$ and $\mathcal{L}*_{\overline{\alpha}}G= (R*_{\alpha}G, R*_{\alpha}G, R*_{\alpha}G)$.
 \end{corollary}

\section{Dimensions of Crossed Products}\

In this section, we study some homological dimensions of partial crossed products. Moreover, we give some characterizations for the partial crossed products to be  symmetric  and Frobenius algebras. We begin with the following proposition. 

    \begin{proposition} Let $R$ be a $K$-algebra,  where $K$ is a commutative ring, $\alpha$ an unital  twisted partial action of a finite group $G$ on $R$ such that  $|G|$ a unit in $R$ and $M$ a left $R*_{\alpha,w}G$-module. If $N$ is a submodule of $M$ such that $N$ is a direct summand of $M$ as 
    $R$-module, then $N$ is a direct summand as $R*_{\alpha,w}G$-module.\end{proposition}
    
    \begin{proof} Let $\pi:M\rightarrow N$ be the natural projection as $R$-module. We define  $\Psi:M\rightarrow N$ by $\Psi(v)=\frac{1}{|G|}\sum_{g\in G}w^{-1}_{g^{-1},g}1_{g^{-1}}\delta_{g^{-1}}\pi(1_g\delta_gv)$. It is not difficult to see that   $\Psi$ is an homomorphism of left $A*_{\alpha,w}G$-modules and $\Psi(\alpha)=\alpha$, for all  $\alpha\in N$. So, the result follows. \end{proof}
    
    The following definition is well-known.
    
    \begin{definition} Let  P be an $R$-module. We say that $P$ is projective if and only if for every surjective module homomorphism $f : N\rightarrow M$ and every module homomorphism $g:P\rightarrow M$, there exists a homomorphism $h:P\rightarrow N$ such that $fh=g$. \end{definition}

    \begin{proposition} Let $\alpha$ be an unital  twisted partial action of a group $G$ on a ring $R$ and  $P$ a left $R*_{\alpha,w}G$-module.  If  $P$ is projective as  left $R*_{\alpha,w}G$-module, then   $P$ is projective as left $R$-module.\end{proposition}
    
    \begin{proof}  
    
    Let $0\rightarrow K\rightarrow F\rightarrow M\rightarrow 0$  be an exact sequence of $R*_{\alpha,w}G$-modules such that $F$ is free. Since $M$ is projective as $R$-module, then this sequence splits as $R$-modules. Hence,   by Proposition 52  we have that this sequence splits as $R*_{\alpha,w}G$-modules. So, $M$ is a projective as left $R*_{\alpha,w}G$-module.\end{proof}
    
    A projective resolution $0\rightarrow X_n\rightarrow X_{n-1}\rightarrow \cdots \rightarrow  X_0\rightarrow  M \rightarrow 0$ of the left $R$-module $M$ is said to be of length $n$, where $X_i$ are projectives $R$-module for all $i\in \{0,\cdots,n\}$. The smallest such $n$ is called the projective dimension of $M$, denoted by $pd_R M$ (if  $M$ has no finite projective resolution, we set $pd_R M =\infty$.) In this case the  left global dimension of $R$ is $sup=\{pd M: M \,\, is \,\, a \,\, left \,\, R-module\}$. 
    
\begin{theorem}  Let $\alpha$ be  an unital  twisted partial action of  a finite group $G$ on a $K$-algebra $R$ such that $|G|^{-1}\in R$. Then $lgdim R=lgdim(R*_{\alpha,w}G)$.\end{theorem}

\begin{proof} By Proposition 54 we have that $lgdim R*_{\alpha,w}G\leq lgdimR$. Using (\cite{mcconnel robson}, 7.2.8)  we have that $ lgldim R\leq lgldim (R*_{\alpha,w}G)+pd(R*_{\alpha,w}G)$. Since $R*_{\alpha,w}G$ is a left free $R$-module, then $pd(R*_{\alpha,w}G)=0$. So, $lgdim(R)\leq lgdim(R*_{\alpha,w}G)$. Therefore, $lgdim(R)=lgdim(R*_{\alpha,w}G)$ \end{proof}   

According to \cite{mcconnel robson}, a ring $S$ is said to be hereditary if all  $S$-modules have projective resolution of  lenght at most 1. Moreover, a ring $S$ is said to be semi-simple if  for any left ideal of $R$ is a direct summand of $R$ as left $R$-module. By (\cite{mcconnel robson}, 7.2.7) any semi-simple ring  has left global dimension  zero and any hereditary algebra has left global dimension 1.

The proof of the following result is direct consequence of the last theorem.

\begin{corollary}  With the same assumptions of Theorem 55, the following statements hold.

(i) $R$ is hereditary if and only if $R*_{\alpha,w}G$ is hereditary.

(ii) $R$ is a semisimple artinian ring if and only if $R*_{\alpha,w}G$  is a semisimple artinian ring. \end{corollary}

To study the weak global dimension we need the following three definitions that appears in \cite{mcconnel robson}.

\begin{definition} Let $M$ be a right $R$-module. Then $M$ is said to be flat if $M\otimes _{R}$ is an exact functor.\end{definition}

\begin{definition} Let $M$ be a right $R$-module. The flat dimension $fd M_R$ of the module $M$ is defined as the shortest lenght of a flat resolution of $M_R$ \begin{center} 0$\rightarrow F_n\rightarrow \cdots\rightarrow F_1\rightarrow F_0\rightarrow M\rightarrow 0$,\end{center} where  $F_i$, $i\in \{1,\cdots,n\}$ are flat modules.\end{definition}

\begin{definition} The weak global dimension of a ring $R$ is $w.dim R=sup\{fdM: M\,\, is\,\, a\,\, right \,\,R-module\}$ \end{definition}  

\begin{theorem} Let $\alpha$ be an unital  twisted partial action of s group $G$ on $R$ such that $|G|$ is a unit in $R$. Then $w.dim(R*_{\alpha,w}G)=w.dim(R)$. \end{theorem}

\begin{proof} By (\cite{mcconnel robson}, Lemma 7.2.2) all flat left $R*_{\alpha,w}G$-modules are flat $R$-modules and we get that $w.dim R*_{\alpha,w}G\leq w.dim R$. Since by (\cite{mcconnel robson}, 7.2.8), \begin{center} $w.dim R\leq w.dim R*_{\alpha,w}G+fd_R R*_{\alpha}G$ \end{center}  and $R*_{\alpha,w}G$ is a flat $R$-module because of being  a free $R$-module, then  we have that $w.dim R\leq w.dim R*_{\alpha,w}G$\end{proof}

According  to (\cite{asky}, Theorem IV.2.1) a finite dimensional $K$ algebra $R$ is said to be a Frobenius algebra  if there exists a $K$-linear form $\varphi:A\rightarrow K$  such that $ker\varphi$ does not contain a nonzero left ideal of $R$.

\begin{theorem} \label{frobenius}Let  $\alpha$ be an unital twisted partial action of a finite group $G$ on a finite dimensional $K$-algebra $R$, where $K$ is a field. If $R$ is Frobenius, then $R*_{\alpha,w}G$ is  Frobenius.\end{theorem}
\begin{proof}

We consider the natural projection $\pi:R*_{\alpha,w}G\rightarrow R$ that is an homomorphism of left $R$-modules. By assumption there exists a nonzero form $f:R\rightarrow K$ and we get $f\circ \pi$ is a nonzero form. Let $I$ be a nonzero left ideal of $R*_{\alpha,w}G$ that is contained in $ker (f\circ \pi)$. Thus, there exists a nonzero element $\eta=\sum a_g\delta_g\in I$ and we may assume that $\eta=a_e+\sum_{g\neq e }a_g\delta_g$. Hence, $0\neq \pi(I)\neq ker f$ which contradicts the fact that $R$ is Frobenius. So, $R*_{\alpha,w}G$ is a Frobenius algebra. 
\end{proof}

According to (\cite{asky}, Theorem 2.2)   a finite dimensional $K$-algebra $R$ is a symmetric algebra if there exists $K$-linear form $\varphi:R\rightarrow K$ such that 
$\varphi(ab)=\varphi(ba)$, for all $a,b\in R$ and $ker \varphi$ does not contain a nonzero one-sided ideal of $R$.

The proof of the following result follows the same ideas of Theorem \ref{frobenius}.

\begin{theorem} Let  $\alpha$ be an unital twisted partial action of a finite group $G$ on a finite dimensional $K$-algebra $R$, where $K$ is a field. If $R$ is  symmetric, then $R*_{\alpha,w}G$ is  symmetric.
\end{theorem}

 \hspace{.6cm}

\end{document}